\documentclass[10pt,a4paper]{amsart}
\usepackage[utf8]{inputenc}
\usepackage{amsmath,amsthm,enumitem}
\newtheorem{theorem}{Theorem}[section]
\usepackage{amsfonts}
\newtheorem{lm}[theorem]{Lemma}
\newtheorem{thm}[theorem]{Theorem}

\newtheorem{prop}[theorem]{Proposition}
\newcommand{\abs}[1]{\left\lvert #1 \right\rvert}
\def\F{\mathbb{F}}
\def\E{\mathcal{E}}

\def\R{\mathbb{R}}
\def\Z{\mathbb{Z}}
\def\N{\mathbb{N}}

\newcommand{\pft}[1]{\left\{ \theta  #1 \right\}}
\newcommand{\nor}[1]{\left\lVert #1 \right\rVert}
\def\T{\mathbb{T}}
\def\Q{\mathbb{Q}}
\def\D{\mathcal{D}}
\newcommand\dens[1]{\mathrm{d}\big(#1\big)}
\newcommand\ldens[1]{\underline{\mathrm{d}}\big(#1\big)}
\def\lmu{\underline{\mu}}
\usepackage{amssymb}
\theoremstyle{definition}

\newtheorem{ex}[theorem]{Example}

\title[On the density or measure of sets and their sumsets]{On the density or measure of sets and their sumsets in the integers or the circle}
\author[P.-Y. Bienvenu]{Pierre-Yves Bienvenu}
\address{P.-Y. Bienvenu, Univ Lyon,  CNRS, ICJ UMR 5208, 
69622 Villeurbanne cedex, France}
\email{pbienvenu@math.univ-lyon1.fr}
\author[F. Hennecart]{Fran\c cois Hennecart}
\address{F. Hennecart, Univ Lyon, UJM-Saint-\'Etienne, CNRS, ICJ UMR 5208, 42023 Saint-\'Etienne, France}
\email{francois.hennecart@univ-st-etienne.fr}

\thanks{This work was performed within the framework of the LABEX MILYON (ANR-10-LABX-0070) of Universit\'e de Lyon, within the program ``Investissements d'Avenir" (ANR-11-IDEX-0007) operated by the French National Research Agency (ANR)}


\begin{document}
\begin{abstract}
Let $\dens{A}$ be the asymptotic density (if it exists) of a sequence of integers $A$.
For any real numbers $0\leq\alpha\leq\beta\leq 1$, we solve the question of the existence of a sequence $A$ 
of positive integers such that $\dens{A}=\alpha$ and $\dens{A+A}=\beta$.
More generally we study the set of $k$-tuples
$(\dens{iA})_{1\leq i\leq k}$ for $A\subset \Z$.
This leads us to introduce subsets defined
by diophantine constraints inside a
random set of integers known as the set of ``pseudo $s$th powers''.
We consider similar problems for subsets of the circle
$\R/\Z$, that is,
we partially determine
the set of $k$-tuples 
$(\mu(iA))_{1\leq i\leq k}$ for $A\subset \R/\Z$.
\end{abstract}
\maketitle

\section{\bf Introduction}

For $A\subset\mathbb{N}$ and $t>1$, we let $A(t)=|A\cap[1,t]|$. We define if it exists the so-called 
\emph{asymptotic density}
of $A$ by
$$
\dens{A}=\lim_{t\to\infty}\frac{A(t)}{t}.
$$ 
Otherwise we define the \emph{lower} and the \emph{upper} asymptotic densities $\underline{\mathrm{d}}(A)$ and
$\overline{\mathrm{d}}(A)$ using $\liminf$ and
$\limsup$ instead of limits.
More generally, if $A\subset B\subset \N$, we
define if it exists the density of $A$ inside $B$
as
$$
\mathrm{d}_B(A)=\lim_{t\to\infty}\frac{A(t)}{B(t)}.
$$
The density of $A$ inside $\N$ is therefore simply
the density, and
if $B$ has a density, we have
$\mathrm{d}_B(A)=\dens{A}/\dens{B}$.

\medskip
For a subset $A$ of a semigroup $G$, let $A+A=\{a+b : a,b\in A\}$.
For $k\geq 1$, we denote by $kA$ its
$k$-fold sumset.
From Kneser's Theorem \cite{Kn}, we know that for
subsets $A\subset \N$, the inequality $\ldens{A+A}<2\ldens{A}$ may only hold when
$\ldens{A+A}$ is a rational number. 
Similarly, for any subset $A$ of the circle  
$\T=\R/\Z$ equipped with its Haar probability measure $\mu$, 
a theorem of Raikov \cite{Ra} implies that $\lmu(2A)\geq \min(1,2\lmu(A))$ where 
$\lmu(A)=\sup_{\substack{F\subset A\\F\text{ closed}}}\mu(F)$.

In this paper, we determine the possible values 
$(\alpha,\beta)$ of pairs $(\ldens{A},\ldens{2A})$ and $(\lmu(A),\lmu(2A))$. We first completely settle the case $\beta \geq \min (1,2\alpha)$.

\begin{theorem}
\label{deuxDensites}
Let $(\alpha,\beta)\in [0,1]^2$. Suppose $\beta\geq \min(2\alpha,1)$. Then the following statements both hold.
\begin{enumerate}
\item[\rm a)] There exists $A\subset\N$ such that $\dens{A}$ and $\dens{2A}$ exist and equal $\alpha$ and $\beta$ respectively.
\item[\rm b)] There exists a measurable subset $A\subset \T$ such that $2A$ is measurable and
$\mu(A)=\alpha$ and $\mu(2A)=\beta$.
Further, for $\alpha >0$, we may take $A$ to be open (in fact a finite union of open intervals).
\end{enumerate}
\end{theorem}

The case $\beta=2\alpha$ is obvious for the second item (with an interval $A$), and is a special case of a theorem by Faisant \textit{et al} \cite{Fa} for the first item, whereas allowing different summands,
Volkmann \cite{Vo} proved that, given positive real numbers 
$\alpha_1$, $\alpha_2$ and $\gamma$ such that
$\alpha_1+\alpha_2\le \gamma<1$, there exist\footnote{In Volkmann's construction, the sets of integers are sets of relative integers and not necessarily positive integers though.} $A_1,A_2$ such that $\dens{A_i}=\alpha_i$, $i=1,2$, and $\dens{A_1+A_2}=\gamma$; he actually proved the corresponding result for subsets of the circle too. A similar result was obtained by Nathanson \cite{Nath}, including a version for Schnirelmann's density.

More generally, we investigate the set $\D_k$ of possible values of the tuple
$$(\dens{A},\dens{2A},\ldots,\dens{kA})$$ 
when $A$ ranges over the set of sequences
for which all of these densities exist.
In parallel, we consider the similar problem in the circle
$\T=\R/\Z$ equipped with its Haar measure $\mu$.
Thus let $\E_k$ be the set of all the possible values of
$(\mu(A),\ldots,\mu(kA))$ for $A\subset \T$ for which these measures exist.
We may sometimes need to work with the subset $\E_k^o\subset \E_k$
of all the possible values of
$(\mu(A),\ldots,\mu(kA))$ for $A\subset \T$ open
and Riemann-measurable
and similarly $\E_k^c$, where we consider closed sets $A$.

There is a close connection between $\E_k$ and $\D_k$
because of Weyl's criterion for equidistribution, of which we now state a direct consequence.
For $A\subset \T$ and $\lambda \in \R\setminus \Q$, let  
$B_{\lambda, A}=\{n\in\N : \{\lambda n\}\in A\}$, where $\{x\}=x-\lfloor x\rfloor$ denotes the fractional part of the real number $x$.
\begin{theorem}
\label{Weyl}
For any irrational number $\lambda$
and any Riemann-measurable function $f : [0,1]\rightarrow\R$, we have
$$
\lim_{x\rightarrow +\infty}\frac{1}{x}\sum_{n\leq x}f(\{\lambda n\})=\int f.
$$
In particular, for any Riemann-measurable
subset $A\subset \T$,
we have
$\dens{B_{\lambda, A}}=\mu(A)$.
The latter equality may be extended to open sets $A$.
\end{theorem}
The extension to open sets is \cite[Lemma 4]{Vo}.
Further, Theorem \ref{Weyl} and a simple compactness argument shows that for any $\epsilon >0$, there exists a constant $C=C(\theta,\epsilon)$ such that
 for any interval $I$ of length at least $\epsilon$ we have $B_{\theta,I}(C)\geq 1$.
Finally, the operation $A\mapsto B_{\lambda,A}$
behaves well with respect to set addition.
\begin{lm}
\label{OpenSum}
Let $k\geq 2$ and $A_i\subset \T$ be open for $i=1,\ldots,k$ and $\lambda$ irrational.
Then $\dens{B_{\lambda, \sum_{i=1}^k A_i}}=
\mu(\sum_{i=1}^kA_i)$.
\end{lm}
\begin{proof}
For $A\subset\T$ open, let
$A^\epsilon=\{x\in A : \textrm{dist}(x,\partial A)>\epsilon\}$.
Thus $A=\bigcup_{\epsilon>0} A^\epsilon$ and $\mu(A)=\lim_{\epsilon\rightarrow 0}\mu(A^\epsilon)$.
Further, $\sum_i A_i=\bigcup_{\epsilon >0} \sum_i A_i^\epsilon$.
We observe that 
\begin{equation}
\label{eq:encadrement}
B_{\lambda,\sum_i A_i^\epsilon}\subset \sum_i B_{\lambda,A_i}\subset B_{\lambda,\sum_i A_i}.
\end{equation}
The rightmost inclusion is easy; for the leftmost one, let
$x\in B_{\lambda,\sum_i A_i^\epsilon}$, thus
$x=\sum_i a_i$ where
$a_i\in A_i^\epsilon$. Consequently,
$(a_i-\epsilon/k,a_i+\epsilon/k)\subset A_i$ for $i\in \{1,\ldots,k-1\}$.
If $n$ is large enough (larger than some constant $C(\epsilon,k)$),
there exists $n_1,\ldots,n_{k-1}\leq n/k$ such that 
$\{n_i\lambda\}\in (a_i-\epsilon/k,a_i+\epsilon/k)$.
Let $n_k=n-n_1-\cdots-n_{k-1}> 0$. Then
$\{n_k\lambda\}=\{n\lambda\}-\{n_1\lambda\}-\cdots -\{n_{k-1}\lambda\} \text{ mod }1$, which implies
$\{n_k\lambda\} \text{ mod }1\in (a_k-\epsilon,a_k+\epsilon)\subset A_k
$, in other words $n_k\in B_{\lambda,A_k}$.
Thus $n\in \sum_i B_{\lambda,A_i}$.

Taking densities and applying Theorem \ref{Weyl} in equation \eqref{eq:encadrement}, we find that
$$
\mu\Big(\sum_is A_i^\epsilon\Big)\leq 
\underline{\mathrm{d}}\Big(\sum_i B_{\lambda,A_i}\Big)\leq\overline{\mathrm{d}}\Big(\sum_i B_{\lambda,A_i}\Big)\leq
\mu\Big(\sum_i A_i\Big).
$$
Letting $\epsilon\rightarrow 0$, we conclude.
\end{proof}
Consequently, $\E_k^o\subset \D_k$; in particular, the second item of Theorem \ref{deuxDensites} implies the first one when $\alpha >0$, but we will provide another proof for it.
Further, Raikov's theorem together with Theorem \ref{deuxDensites}
means that $\E_2=\E_2^o=\E_2^c=\{(\alpha,\beta)\in [0,1]^2 : \beta\geq \min(1,2\alpha)\}$.

To complete our description of $\D_2$, we need to understand
which pairs $(\alpha,\beta)$ with $\beta<2\alpha$ belong to it,
which we do in the next theorem. For an integer $n$, let $v_2(n)$ be
its dyadic valuation; we extend it to rational numbers by letting
$v_2(p/q)=v_2(p)-v_2(q)$.

\begin{thm}
\label{rationnel}
Let $\beta\in\mathbb{Q}\cap(0,1)$ such that $v_2(\beta)\le0$, let $\alpha\in(0,1)$ satisfy $\beta < 2\alpha$ and $g_0$ denote
$\min\{g\ge 1\,:\, g\beta\text{ is odd}\}$. Then there exists a sequence $A\subset\mathbb{N}$ such 
that $\dens{A}=\alpha$, $\dens{2A}=\beta<2\alpha$ if and only if
$$
 \frac{\beta}2<\alpha\le \frac{\beta}2+\frac1{2g_0}.
$$ 
\end{thm}

\begin{ex}
The pair $\alpha=4/9$, $\beta=5/9$ enforces $g_0\le3$ and $1\le r\le 2$, whence
$\beta=1$, $1/2$ or $1/3$, a contradiction.
\end{ex}

\begin{ex}
The pair $\alpha=1/5$, $\beta=3/10$ yields $g_0\le10$ and $1\le r\le 5$. Choosing $r=2$
gives the required condition.
\end{ex}

We briefly discuss iterated sumsets. It is not clear what constraints
a tuple $(\alpha_i)_{i\in [k]}$ must satisfy for a set $A\subset \R/\Z$ satisfying $\mu(iA)=\alpha_i$ to exist; we certainly need
$\alpha_i\geq \min(1,\alpha_j+\alpha_{i-j})$ for any $j<i$ due to Raikov's theorem but it may not be sufficient.
In particular,  we will deduce the following constraint from a theorem of Gyarmati, Konyagin and Ruzsa \cite{Gy}.
\begin{theorem}
\label{th:Gy}
There exists a constant $c>0$ such that the following holds.
Let $A\subset \T$ be closed, and suppose that $\mu(2A)<c$. Then $\mu(3A)\geq \frac32\mu(2A)$.
\end{theorem}
In view of Lev's analogous result \cite{Lev} on finite sets of integers,
one may more generally imagine that
$\mu((k+1)A)\geq \frac{k+1}{k}\mu(kA)$
under certain restrictions on $\mu(kA)$.
Note that another result from \cite{Gy}
implies that the constant $c$ may 
not be taken to be 1. Gyarmati \emph{et al.} conjecture
that its optimal value is 1/2.
Note that for any finite set $A$ of integers, we have
$2\abs{3A}\geq 3\abs{2A}-1$.
On the other hand, due to the Plünnecke-Ruzsa inequalities,
we know that if $\dens{2A}\leq K\dens{A}$, we must have
$\dens{3A}\ll K^3\dens{A}$.
Similarly, in the circle, if $\mu(2A)<3\mu(A)$ and $\mu(A)$ is small
enough, Moskvin \textit{et al.} \cite{Bi} showed that $A$ must satisfy strict structural conditions that imply that $\mu(3A)\leq 3(\beta-\alpha)$.

We solve partially the problem with $k=3$.
\begin{theorem}
\label{triplets}
Let $(\alpha,\beta,\gamma) \in (0,1]^3$, and suppose that $\beta < \min(3\alpha,1)$
and $\gamma \in [\min(1,3\beta/2),\min(1,2\beta-\alpha)]$ or that
$\beta =3\alpha$ and $\gamma \in [3\beta/2,2\beta]$.
Then $(\alpha,\beta,\gamma)\in \E_3$.
\end{theorem}
For general $k$, our understanding of $\E_k$ and $\D_k$ is yet poorer. Note that in general, our sets $A\subset\N$ satisfy
$\dens{(k+1)A}\geq \frac{k+1}{k}\dens{kA}$,
which, in view of the aforementioned result of Lev, may be inevitable.
\begin{theorem}
\label{generalk}
Let $\alpha=(\alpha_1,\ldots,\alpha_{k+1})\in [0,1]^{k+1}$, where $k\geq 1$.
\begin{enumerate}[label=\rm \alph*)]
\item If $\alpha_1=\cdots=\alpha_{k-1}=0$ and $\alpha_{k+1}\geq \frac{k+1}{k}\alpha_{k}$, or $\alpha_{k+1}\geq \alpha_k$ and
$\alpha_{k+1}$ is the inverse of an integer, then $\alpha\in \D_{k+1}$.
\item If $\alpha_1=\cdots=\alpha_{k}=0$, then $\alpha\in \E_{k+1}$.
\item If $\alpha_i=i\alpha$ for each $i$ and some $\alpha\leq 1/(k+1)$, then $\alpha\in \E_{k+1}^0\subset \D_{k+1}$.
\end{enumerate}
\end{theorem}
The last item is obvious by taking an interval of length $\alpha$,
and was also proven somewhat differently for $\D_k$ in \cite{Fa}.

In the next section, we prove the complete description of
$\D_2$ and $\E_2$ given in Theorems \ref{deuxDensites} and \ref{rationnel}. 

\section{\bf Sumsets in the integers}\label{S2}
\subsection{A preliminary reduction}
We show that Theorem \ref{generalk} a)
follows from the special case below, where $\alpha_{k+1}=1$ in the notation of that theorem.

For a real number $\theta >1$, let \begin{equation}\label{defT}
T_{k,\theta}=\Big\{n\ge1\,\mid\, 0<\{\theta n\}<\frac1{k+1}\Big\}.
\end{equation}
Note that $\dens{T_{k,\theta}}=1/(k+1)$ if $\theta$ is irrational,
while $T_{k,\theta}=\N$ if $\theta$ is an integer.
In any case, $(k+1)T_{k,\theta}=\N$.
\begin{prop}
\label{prop:specialcase}
Let $\beta\in [0,1)$ and integer $k\geq 1$.
There exists a set $A\subset T_{k,\theta}$ such that $iA$ has
density 0 for any $i <k$, whereas $kA$ has density $\beta$
inside $kT_{k,\theta}$ and $(k+1)A$ has density 1 in $\N$.
\end{prop}
In particular, we have $\dens{kA}=\beta k/(k+1)$
if $\theta$ is irrational while $\dens{kA}=\beta$ if $\theta$ is an integer.

We now deduce Theorem \ref{generalk} a) from
Proposition \ref{prop:specialcase}.
Let $\alpha \in [0,1]^{k+1}$ be as in the hypothesis of the former theorem,
and let $\beta'=\alpha_k$ and $\gamma'=\alpha_{k+1}$.
We distinguish several cases.
 
 \begin{enumerate}
 \item[a)] We first assume that $\gamma'$ is an irrational number.
 Let $A$ be the set given  in Proposition \ref{prop:specialcase}
 with parameters $\theta=\frac1{\gamma'}$ and $\beta=\frac{\beta'}{\gamma'}$
 and $A'$ be defined by 
 $$
 A'=\{\lfloor \theta a\rfloor,\ a\in A\}.
 $$
Since $A\subset T_{k,\theta}$ we have 
$$
\forall\, a_1,\dots,a_{k+1}\in A, \quad \lfloor\theta a_1\rfloor+\lfloor\theta a_2\rfloor+\cdots +\lfloor\theta a_{k+1}\rfloor=\lfloor\theta(a_1+a_2+\cdots+a_{k+1})\rfloor.
$$
Since $\theta>1$, we get $\dens{j A'}=\theta^{-1} \dens{jA}$, $j=1,2,\dots,k+1$. \\
This yields
Theorem \ref{generalk} a) when $\gamma'$ is an irrational number.

\medskip
\item[b)] If $\gamma'$ is the inverse of a positive integer $q$, we use again
Proposition \ref{prop:specialcase}
 with parameters $\theta=\frac1{\gamma'}$ and $\beta=\frac{\beta'}{\gamma'}$ to generate a set $A$ 
and define a set $A_q=\{qa,\ a\in A\}$ satisfying $$\dens{(k-1)A_q}=0<\dens{kA_q}=\frac{\beta}{q}<\dens{(k+1)A_q}=\frac1q.$$
 
 \medskip
\item[c)] We finally assume that $\gamma'=\frac{s}{q}$ is a rational number with $2\le s < q$. 
Upon multiplying numerator and denominator by appropriate numbers,
we may assume that $s=(k+1)r$ for some integer $r$
satisfying $3\le r < \frac q{k+1}$.
Let
$
U=\{0,1,\dots,r-2,r\}.
$
Then $|jU|=jr$ for any $j$. Letting $A'=U+A_q$, we thus obtain
$$\left\{
\begin{aligned}
&\dens{(k-1)A'}=|(k-1)U|\times \dens{(k-1)A_q}=0,\\ 
&\dens{kA'}=|kU|\times \dens{kA_q}=\frac{kr\beta}{q}=\frac{k}{k+1}\beta \gamma ',\\ 
&\dens{(k+1)A'}=|(k+1)U|\times \dens{(k+1)A_q}=\frac{(k+1)r}{q}=\gamma'.
\end{aligned}\right.
$$
\qedhere
\end{enumerate}

This concludes the proof of Theorem \ref{generalk} a),
assuming Proposition \ref{prop:specialcase}.
We will now prove the latter, focussing first on the case $k=1$ (so concerning twofold sumsets, that is Theorem \ref{deuxDensites}), since it is much more simple
than, while retaining some important features of, the general case, which we handle later.

\subsection{Twofold sumsets}

Before embarking on the proof of Proposition \ref{prop:specialcase}
in the case $k=1$, we need a quantitative version
of Weyl's criterion (Theorem \ref{Weyl}), due to 
Erd\H{o}s and Tur\'an \cite[Theorem III]{ET}.
\begin{theorem}
\label{th:ET}
For any sequence $s_j$
of elements of the torus
and any interval $A$, we
have for any integers $n$ and $m$ the bound
$$
\abs{\frac{1}{n}\abs{\{1\leq j\leq n : s_j\in A\} }- \mu(A)}\ll
\frac{1}{m}+\frac{1}{n}\sum_{k=1}^m\frac{1}{k}
\abs{\sum_{j=1}^ne^{2i\pi s_jk}},
$$
where the implied constant is absolute.
\end{theorem}
Applying this with $s_j=\{\theta j\}$ for some
irrational number $\theta$ and
using the standard exponential sum bound
$$
\abs{\sum_{j=1}^m e^{2i\pi j\theta}}\leq \frac{1}{2\nor{\theta}},
$$
where $\nor{\theta}=\min_{k\in\mathbb{Z}}|\theta -k|$,
we obtain
$$
\abs{\frac{1}{n}B_{\theta,A}(n) - \mu(A)}\ll
\frac{1}{m}+\frac{1}{n}\sum_{k=1}^m\frac{1}{k\nor{\theta k}}.
$$
The series  $\sum_{k=1}^m\frac{1}{k\nor{\theta k}}$
diverges as $m$ tends to infinity, but selecting $m=m(n)$
as a  sufficiently slowly increasing function of $n$,
one may achieve
$$
\frac{1}{n}\sum_{k=1}^{m(n)}\frac{1}{k\nor{\theta k}}\asymp \frac{1}{m(n)}\rightarrow 0
$$
as $n$ tends to infinity,
and thus
there exists a function $\eta : \N\rightarrow\R_+$ (depending on $\theta$ only) that tends to zero such that
\begin{equation}
\label{eq:Weyl}
\abs{\frac{1}{n}B_{\theta,A}(n) - \mu(A)}\leq \eta(n).
\end{equation}
Note that the bound \eqref{eq:Weyl} is uniform in $A$; in particular, it is still valuable if $A$ is replaced by a sequence $A_n$ of intervals
 of sufficiently slowly decaying measure ({e.g. $\mu(A_n)\geq 2\eta(n)$).
 Also we note that using
 the sequence $s_j=\pft{(j+X)}$,
 we may obtain the more general bound
 \begin{equation}
\label{eq:Weyl2}
\abs{\frac{1}{n}(B_{\theta,A}(n+X)-B_{\theta,A}(X)) - \mu(A)}\leq \eta(n)
\end{equation}
for any integers $X$ and $n$.

%
%

\medskip
We now start the proof of Proposition \ref{prop:specialcase} in the case $k=1$.
We will adopt a probabilistic construction.
Let $\theta$ be an irrational number and $\eta$ be a function for which equation \eqref{eq:Weyl} holds and
$$
X_\theta=\Big\{n\in\mathbb{N}\,:\, 2\eta(n/2)<\{\theta n\} <1-2\eta(n/2)\Big\}.
$$
Equation \eqref{eq:Weyl} and the ensuing remarks imply that
\begin{equation}\label{eqdensX}
\dens{X_\theta}=1.
\end{equation}

We now define our  desired random sequence $A$. Let $(\xi_k)_{k\ge 1}$ be a sequence of mutually  independent Bernoulli random variables such that
$$
P(\xi_k=1)=\beta_k,\quad k\ge1
$$
where $\beta_k$ is the constant sequence equal to $\beta$
if $\beta >0$ and  the decaying sequence $k^{-1/5}$ if $\beta=0$.
Let $A$ be the random sequence consisting of the integers $k\in T_{1,\theta}$ such that $\xi_k=1$. It is easy to 
see that the  density of $A$ inside $T_{1,\theta}$  satisfies $\mathrm{d}_{T_{1,\theta}}(A)=\beta$ almost surely as required.

Now we prove that $A+A\supset X_\theta\setminus F$, 
where $F$ is almost surely a finite set. This would imply that 
$\dens{A+A}=1$, as desired.
Let $n\in X(\theta)$.
We define 
$$
K_n=\{0<k<n/2\,:\, k\in T_{1,\theta}\cap(n-T_{1,\theta})\},
$$
and
$$
R(n)=\sum_{k\in K_n}\xi_k\xi_{n-k}.
$$
Then by the independence of the $\xi_k$'s
\begin{equation}\label{eqRn0}
P(R(n)=0)=\prod_{k\in K_n}P(\xi_k\xi_{n-k}=0)\leq (1-\beta_n^2)^{|K_n|}\leq \exp(-\abs{K_n}\beta_n^2).
\end{equation}
We now estimate   $\abs{K_n}$ from below. 
By definition $k< n/2$ belongs to $K_n$
if and only if $\{\theta k\}<1/2$ and $\{\theta(n-k)\}<1/2$.

Let $ I=(2\eta(n/2),1/2)$.
Since $n\in X_\theta$, we have $\{\theta n\}\in I\cup (1-I)$.
Suppose for instance 
$\{\theta n\}\in I$, the case
$\{\theta n\}\in 1-I$ being similar. 
Then for any $k$ such that $\{\theta k\}<\pft{n}<1/2$, we have 
$\pft{(n-k)}=\pft{n}-\pft{k}<1/2$.
Thus $k\in K_n$. This means that
$$K_n\supset \left\{0<k<n/2 : \pft{k}<\pft{n}\right\},$$
whence $\abs{K_n}\geq n/2 (\pft{n}-\eta(n/2))\geq \frac{n}{2}\eta(n/2)$  by equation \eqref{eq:Weyl}.
If $\pft{n}\in 1-I=(1/2,1-2\eta(n))$ instead,
it suffices to replace the condition $\pft{k}<\pft{n}$
by $\frac{1}{2}-\pft{k}<1-\pft{n}$ to obtain the same result.

One can choose $\eta(n)$ to be arbitrarily slowly decaying, say $\eta(n)\geq n^{-1/2}$.
This way $\abs{K_n}\gg \sqrt{n}$, so that $\beta_n^2\abs{K_n}\gg n^{1/10}$ and 
 from \eqref{eqRn0} we get
$$
\sum_{n\in X_\theta}P(R(n)=0)<\infty.
$$
We conclude by the Borel-Cantelli lemma (cf. \cite[Lemma 1.2]{TV}) that almost surely, all but finitely many integers of $X_\theta$ are sums of $2$ terms from the random sequence $A$. The result follows from \eqref{eqdensX}. This finishes the proof of Proposition \ref{prop:specialcase} in the case where $k=1$,
and thus 
of Theorem \ref{deuxDensites}.

\bigskip

We now determine which pairs
$(\alpha,\beta)\in \R^2$ with
 $0<\alpha\le \beta<2\alpha$ belong to $\D_2$, that is, we prove Theorem \ref{rationnel}.

Let $A\subseteq\mathbb{N}$ such that $\beta=\dens{2A}<2\dens{A}=2\alpha$. By Kneser's theorem for infinite sequences, there exists a (minimal) positive integer $g$ such that $\ (2A+g\mathbb{N})\setminus2A$ is finite and
$$
\dens{2A}\ge2\dens{A}-\frac1g.
$$
Let
\begin{align*}
A_g&=\{\overline{x}=x+g\mathbb{Z}\in \mathbb{Z}/g\mathbb{Z}\,:\, \overline{x}\cap A\ne\varnothing\},\\
A'_g&=\{\overline{x}\,:\, |\overline{x}\cap A|=\infty\},\\
A''_g&=\{\overline{x}\,:\, 0<|\overline{x}\cap A|<\infty\}.
\end{align*}
We have $A_g=A'_g\cup A''_g$. Let
$$
\widetilde{A}=\bigcup_{\overline{x}\in A'_g}(x+g\mathbb{N})\cup\{x\,:\, \overline{x}\in A''_g\}.
$$
Then
$$
\dens{A}\le \dens{\widetilde{A}}=\frac{|A'_g|}{g},\quad
\dens{2A}=\frac{|A_g+A'_g|}{g}.
$$
Since $g$ is minimal we have $|A_g+A'_g|\ge|A_g|+|A'_g|-1$ since otherwise $A_g+A'_g$ has a nontrivial 
period. Hence from $\dens{2A}<2\dens{A}$ we get $|A_g|+|A'_g|-1 <2|A'_g|$ giving $A'_g=A_g$ and finally $|2A_g|=2|A_g|-1$. 

Let $r=|A_g|$. Then $\beta=\frac{2r-1}{g}$ with $1\le r\le\frac{g+1}2$. We get
$$
\frac{\beta}2<\alpha\le \frac rg=\frac{\beta}2+\frac1{2g}.
$$
We proved the following.
\begin{prop}
Let $A$ such that $\dens{2A}<2\dens{A}$. Then there exist two positive integers $g$ and $r\le \frac{g+1}2$ such  that
$$
\dens{2A}=\frac{2r-1}g\quad\text{and}\quad \frac{\dens{2A}}2<\dens{A}\le \frac{\dens{2A}}2+\frac1{2g}.
$$
\end{prop}
Conversely, let $\beta\in [0,1]\cap \Q$ have nonpositive dyadic
valuation, and $g$ be the smallest positive integer for which $g\beta$ is odd, thus $\beta=\frac{2r-1}{g}$ and let 
$\alpha$ satisfy $$\frac{\beta}2<\alpha\le \frac{\beta}2+\frac1{2g}=\frac{r}{g}$$
Then let $R=\{0,\ldots,r-1\}\in \Z/g\Z$, so that $\abs{2R}=2r-1$.
Let $\gamma=\alpha \frac{g}{r}$, thus $\gamma\in (0,1]$.
Take $A_g\subset g\N$ constructed in the proof of Proposition \ref{prop:specialcase} (with $k=1$), so that
$\dens{A_g}=\gamma/g$ and
$\dens{2A_g}=1/g$ and let
$A=\cup_{x\in R}x+A_g$, which has density $\alpha$.
Consequently $2A=\cup_{x\in 2R}x+g\Z$, which yields $\dens{2A}=\beta$ as desired.

This completes the proof of Theorem \ref{rationnel}.

%

\section{\bf Measures of sumsets in the circle}
\subsection{Twofold sumsets}
To start with, we show that in order to achieve a large ratio 
$\mu(2A)/\mu(A)$, a large number of connected components will be necessary.
\begin{prop}
Let $A$ be a disjoint union of $k$ intervals.
Then $\mu(2A)\leq (k+1)\mu(A)$. If the intervals are open, the equality case happens when all the $\binom{k+1}{2}$ intervals of the sum
are pairwise disjoint.
\end{prop}
\begin{proof}
Let $A=\bigcup_{j=1}^kI_j$. So $A+A=\bigcup_{i\leq j} (I_i+I_j)$.
Let $\mu(I_i)=m_i$, so $\mu(I_i+I_j)=m_i+m_j$ and
$\mu(A+A)\leq \sum_{i\leq j}(m_i+m_j)=(k+1)\sum_i m_i$.
The equality case is clear.
\end{proof}
We now attempt to prove the first item of Theorem \ref{deuxDensites} in the case $\alpha >0$.
Let $(\alpha,\beta)\in (0,1]^2$ satisfy $\beta \geq \min (2\alpha,1)$.
If $\beta=\min(2\alpha,1)$, the set $A=(0,\alpha)$ satisfies
$\mu(A)=\alpha,\mu(2A)=\beta$.
So we now suppose $0<\alpha<1/2$ and $\beta >2\alpha$.

First, note that for any $k$,
if
 $A=[0,\ell]\cup \{2\ell\}\cup \cdots \cup \{(k-1)\ell\}$, then $A+A=[0,k\ell]$ so we can achieve a duplication ratio $\mu(2A)/\mu(A)=k$.
 The idea is then to somewhat
``thicken'' the singletons, in order to reduce the duplication ratio of the set.

Let $k=\lfloor \beta/\alpha \rfloor$, thus $k\leq \beta/\alpha < k+1$ and $k\geq 2$.

Then let $A=(0,x)\cup (\{2x,\ldots,kx\}+(-\epsilon,0))$,
for some $x\leq \alpha$ and $\epsilon \leq x/2$ to be determined later.
Note that $$A+A=(0,(k+1)x)\cup (\{(k+2)x,\ldots,2kx\}+
(-2\epsilon,0)).$$
Thus $\mu(A)=x+(k-1)\epsilon$
and $\mu(2A)=(k+1)x+2(k-1)\epsilon$.
The doubling ratio is therefore 
$$
f(\epsilon/x)=\frac{(k+1)x+2(k-1)\epsilon}{x+(k-1)\epsilon}=
(k+1)\frac{1+2\frac{k-1}{k+1}\frac{\epsilon}{x}}{1+(k-1)\frac{\epsilon}{x}}.
$$
We have $f(0)=k+1$ and while $f(1/2)=4k/(k+1)\leq k$.
Therefore by continuity of $f$, there is a value of the ratio $y=\epsilon/x$ for which the doubling ratio is the desired $\beta/\alpha$.

Then there remains to pick $x$ such that $\alpha=x+(k-1)\epsilon=x(1+(k-1)y)$,
namely $x=\frac{\alpha}{1+(k-1)y}$, and then the corresponding $\epsilon$.

In the case $\alpha=0$, a radically different construction will be necessary. Let $C\subset [0,1]$ be the classical ternary Cantor set.
It is well known that $C+C=[0,2]$. 
For the sake of completeness, we reproduce a short proof.
It suffices to prove $C+C\supset [0,2]$. 
Let $u\in [0,2]$ and let $(\epsilon_i)_{i\geq 1}\in \{0,1,2\}^{\N\setminus\{0\}}$ be the digits of $u/2$ in its ternary expression,
thus
$$
u/2=\sup_{i\ge1}\epsilon_i3^{-i}.
$$
We construct sequences $\alpha$ and $\beta$ in $\{0,2\}^{\N\setminus\{0\}}$
 such a way that
for each $i\geq 1$, we have
$\alpha_i+\beta_i=2\epsilon_i$. Thus
if $\epsilon_i=0$ we take $\alpha_i=\beta_i=0$;
if $\epsilon_i=1$ we define $\alpha_i=0$ and $\beta_i=2$;
otherwise $\alpha_i=\beta_i=2$.
Letting $x=\sum_{i\geq 1}\alpha_i3^{-i}$ and $y=\sum_{i\geq 1}\beta_i3^{-i}$, we see that $x$ and $y$ are in $C$
and $x+y=2 \cdot u/2=u$, which concludes.
 
Scaling $C$ it by a factor $\beta/2$
and projecting it to the circle,
we obtain a set $A=(\beta/2)C$ of measure 0 such that $\mu(2A)=\beta$.

%
%
%
%
%
\subsection{Threefold sumsets}
First we prove Theorem \ref{th:Gy}.
We will derive it from the following theorem of Gyarmati, Konyagin and Ruzsa \cite{Gy}.
\begin{prop}
\label{prop:Gy}
There exists an absolute constant $c>0$ such that the following holds.
Let $p\geq 29$ be a prime. Let $A\subset \Z/p\Z$
and
let $(n,s)=(\abs{2A},\abs{3A})$.
If $n<cp$, then $s\geq \frac{3n-1}{2}$.
\end{prop}
We derive the analogous result for measures in the circle
by a standard method.
We first prove Theorem \ref{th:Gy} for simple sets,
that is, the  union of finitely many closed intervals.
Let $A\subset \T$ be a simple set.
Let $c$ be the constant given by Proposition~\ref{prop:Gy} and
suppose that $\mu(2A)<c
$. 
Let $p\geq 29$ be a prime, that we will let tend to infinity ultimately.
Let
$$
A(p)=\left\{j\in\Z/p\Z : \frac{j}{p}\in A\right\}.
$$
This notation should not conflict with the notation $A(t)$
defined in the introduction.
One may check that $\abs{A(p)}=p\mu(A)+O(1)$
as $p$ tends to infinity.
Further note that
$(kA)(p)=kA(p)$ for any $k\in \N$.
Since $2A$ and $3A$ are simple, one has
$\abs{(kA)(p)}=p\mu(kA)+O(1)$ for $k=2,3$; 
thus
we have $\abs{(2A)(p)}<cp$ for
$p$ sufficiently large, so we can apply
Proposition \ref{prop:Gy} and
 conclude in the case of simple sets.
 
 Now if $A$ is closed (that is, compact),
 writing $I_\delta=(-\delta,\delta)$,
 we have $A=\bigcap_{\delta>0}(A+I_\delta)$,
 in fact
 $kA=\bigcap_{\delta>0}(kA+I_{k\delta})$
 for any integer $k\geq 1$.
 So for any fixed $\epsilon>0$, we can chose $\delta$ such that
 $\mu(kA+I_{k\delta})\leq \mu(kA)+\epsilon$.
 Further, by compacity, there exists a simple set $A'$
 (the union of finitely many translates of $I_\delta$)
 such that $A\subset A'\subset A+I_\delta$.
 We have
 $$\mu(3A)\geq \mu(3A')-\epsilon\geq \frac32\mu(2A)-\epsilon.$$
 Letting $\epsilon$ tend to zero, we conclude
 the proof of Theorem \ref{th:Gy}.


We prove Theorem \ref{triplets}.
If $\alpha\geq 1/3$, the triplets $(\alpha,\beta,\gamma)$ that
belong to $\E_k$ are the ones for which $\beta \geq \min(1,2\alpha)$ and $\gamma=1$.

We now consider triplets where $\alpha<1/3$; we prove the following proposition, which implies Theorem \ref{triplets}.
\begin{prop}
The set of triplets $(\mu(A),\mu(2A),\mu(3A))$ for sets 
$A\subset [0,1/3]\subset \T$ having at most two connected components is 
$$
\{(\alpha,\beta,\gamma)\in [0,1]^3 : \beta\in [2\alpha,3\alpha],\gamma\in [3\beta/2,2\beta-\alpha) \text{ or } \beta=3\alpha,
\gamma\in [3\beta /2,2\beta]\}.
$$
\end{prop}
\begin{proof}
We may take $A$ of the form
$(0,x)\cup (y,z)$ for some $0\leq x\leq y\leq z\leq 1/3$.
So $A+A=(0,2x)\cup (y,x+z)\cup (2y,2z)$
and $3A=(0,3x)\cup (y,2x+z)\cup (2y,2z+x)\cup (3y,3z)$.

We are seeking for which triplets
$(\alpha,\beta,\gamma)$ the system

$$\left\lbrace
\begin{array}{c @{\, = \,} c}
\alpha & x+z-y	\\

\beta &  3\alpha - \max(0,2x-y) -\max(0,x+z-2y) \\ 
 
 \gamma &  6\alpha-\max(0,3x-y)-\max(0,2x+z-2y)-\max(2z+x-3y,0) \\ 
\end{array}
\right.$$
admits solutions.
We now discuss the existence of solutions according to the number of connected components of $2A$ and $3A$, that is, for each max above, whether it is positive or not. In the following discussion,
the necessary conditions we provide may always easily be seen to
be sufficient, although we do not always explicitly state it.

\begin{enumerate}[label=\arabic*)]
\item If $2A$ is an interval, then so is $3A$ so
$\gamma=3\alpha=3\beta/2$.
\item If $2A$ has two connected components, so exactly one overlap between the intervals of $2A$, we distinguish.
\begin{enumerate}[label=\alph*)]
\item If $2x>y$ and $x+z<2y$, so 
$2A=(0,x+z)\cup (2y,2z)$, we have $\beta=x-2y+3z$.
We have necessarily $3x>y$ and $2x+z>2y$, so
$3A=(0,2z+x)\cup (3y,3z)$ where the last two intervals may overlap or not.
\begin{enumerate}[label=\roman*)]
\item If they do, so $2z+x>3y$,
we have $\gamma=3z$.
So $\beta=x-2y+\gamma$ and $\alpha=x-y+\gamma/3$.
Get $\alpha-\beta=y-2\gamma/3$
so $y=\alpha-\beta+2\gamma/3$
while $x=2\alpha-\beta+\gamma/3$.
We check that the inequalities are satisfied:
$2x-y=3\alpha-\beta>0$ so $\beta<3\alpha$,
$2y-x-z=-\beta+2\gamma/3>0$ implies $\gamma>3\beta/2$.
Further, we need $2z+x-3y=-\alpha+2\beta -\gamma>0$
which amounts to $3\beta/2<\gamma<2\beta-\alpha<5\alpha$.
Conversely, whenever these conditions are satisfied, the system has solutions.
\item If they don't,
so $2z+x<3y$, we have $\gamma=3(z-y)+2z+x$.
Thus a solution exists if and only if $\gamma=2\beta-\alpha$.

\end{enumerate}
\item Now if $2x<y$ and $x+z>2y$, so
$2A=(0,2x)\cup (y,2z)$,
we have $\beta=2x-y+2z$.
We have necessarily $2x+z>2y$ and $2z+x>3y$,
so $3A=(0,3x)\cup (y,3z)$, where the two intervals 
may or not overlap.
\begin{enumerate}[label=\roman*)]
\item If they do, so $2x<y<3x$, we have
$\gamma=3z$. Further we find $y=\beta-2\alpha$,
and $x=\beta-\alpha-\gamma/3$.
So $y-2x=-\beta+2\gamma/3>0$ implies yet again
$\gamma>3\beta/2$.
Further $y-3x=-2\beta + \alpha+\gamma<0$ implies 
$\gamma<2\beta-\alpha$.
Also $x+z-2y =3\alpha-\beta >0$ amounts to $\beta <3\alpha$.
\item Otherwise, so $y>3x$,
we find
$\gamma=3z-y+3x=3\alpha+2y$ and again $\gamma=2\beta-\alpha$.
\end{enumerate}

\end{enumerate}

\item If $2A$ has three connected components (no overlap),
then $\beta=3\alpha$.
We have $2x<y$ and $x+z<2y$.
We distinguish according to the presence of overlaps or
not in $3A$.
\begin{enumerate}[label=\alph*)]
\item If there is no overlap, we have $\gamma=6\alpha$. It is realisable, just take $x$, then $y>3x$, then $y<z<\min((3y-x)/2,1/3)$, then all constraints are realised.
We can achieve that for any value of $\alpha\leq 1/6$.
\item If $3A$ is connected,
$\gamma=3z$. 
Now the conditions $2x<y$ and $x+z<2y$ imply $z<3(y-x)$,
which is equivalent to $2z>3(x+z-y)$, and finally
 $\gamma>3\beta/2$.
\item If there is exactly one overlap, that is, if $3A$ has three connected components, we distinguish.
\begin{enumerate}[label=\roman*)]
\item Suppose $3x>y$. 
And $2x+z<2y$ and $2z+x<3y$.
So $\gamma=6\alpha-3x+y$. This imposes $\gamma\in (5\alpha,6\alpha)=(5\beta/3,2\beta)$.
\item 
Now suppose $2x+z>2y$. 
And $y>3x$ and $2z+x<3y$.
Then $\gamma=6\alpha-2x-z+2y=5\alpha-x+y>5\alpha$.
\item If only the last gap is overcome,
$\gamma=6\alpha-2z-x+3y=5\alpha-z+2y>5\alpha$.
\end{enumerate}
\item If $3A$ has two connected components, we distinguish.
\begin{enumerate}[label=\roman*)]
\item If all but the last gap are overcome, 
$\gamma=6\alpha-3x+y-2x-z+2y=5\alpha-4x+2y>5\alpha.$
\item If all but the middle gaps are overcome,
$\gamma=6\alpha-3x+y-2z-x+3y=5\alpha-3x-z+3y>5\alpha $.
\item If all but the first gap are overcome,
$\gamma=6\alpha-2x-z+2y-2z-x+3y>5\alpha$.\qedhere
\end{enumerate}
\end{enumerate}
\end{enumerate}
\end{proof}

Regarding sets with $k$ connected components when $k\geq 3$, the determination of the possible triplets
$(\alpha,\beta,\gamma)$ becomes untractable by this method. Nevertheless, we can easily
see that the structure of the set of the possible
triplets remains similar, that is, a connected union of finitely many (in fact $O_k(1)$ many) polytopes, where a polytope is the intersection of finitely many half-spaces.

\subsection{Further iterated sumsets}
We now prove Theorem \ref{generalk} b).
Let $\beta \in (0,1]$ and $k\geq 2$ an integer.
Let $C_{k+1}$ be the Cantor set of initial segment 
$[0,1]\subset\R$ and ratio
of dissection $1/(k+1)$.
It is known \cite[Corollary 2.3]{Ca} that $\mu((k-1)C)$ has measure 0 whereas 
$kC=[0,k]$.
A suitable scaling of $C_{1/(k+1)}$ provides the desired construction.

Note that this does not imply the first point of
Theorem \ref{generalk}: the openness condition of Lemma \ref{OpenSum} may not be removed.
Indeed, if $A\subset \R/\Z$ has measure zero, one may see that
$B_{\lambda,A}$ is empty for almost all $\lambda\in \R/\Q$, since
the map $\lambda \mapsto n\lambda$ on the circle is measure-preserving for any integer $n$. So we need to provide a specific proof, which we do in the next section.

\section{\bf Iterated sumsets in the integers}
We now prove
Proposition \ref{prop:specialcase} for $k\geq 2$.
The (probabilistic) argument we will use subsumes, but is significantly
more complicated than, the one used in Section~\ref{S2},
which is why we preferred to present it separately.
First of all we collect a number of useful but technical results.
\subsection{Preliminary lemmas}
First we need to somewhat generalise the bound \eqref{eq:Weyl2} obtained via the Erd\H{o}s-Tur\'an theorem.
\begin{prop}
\label{prop:genET}
Let $k,D,M,X$ be integers.
Let
$f=\sum_{i=1}^kP_i \mathbf{1}_{I_i}$
where  $(I_i)_{i\leq k}$ is a family of pairwise disjoint intervals in $[0,1)$ and $P_i$ a polynomial of degree less than $D$ whose coefficients are all at most $M$.
Then $$
\abs{\frac{1}{N}\sum_{X<n\leq N+X}f(\{\theta n\})-\int f}=O( MDk\sqrt{\eta(N)}).
$$
\end{prop}
A function $f$ satisfying the above hypothesis will naturally be referred to as \textit{piecewise polynomial}.
\begin{proof}
It suffices to prove it for monomials and for $k=1$, the
general case following by linear combinations (incurring an extra factor $Mk$).
Thus let $a<b$ be in $[0,1)$, and let $d\leq D$ and $f$
be defined by $f(x)=x^d\mathbf{1}_{(a,b)}$.
Using
the bound \eqref{eq:Weyl2},
we note that
$$
a^d ((b-a)-O(\eta(N))\leq\frac{1}{N}\sum_{X<n\leq N+X}\{\theta n\}^d\mathbf{1}_{(a,b)}(\pft{n})
\leq b^d ((b-a)+O(\eta(N))
$$
Further, observe that 
$$
a^d(b-a)\leq\int_a^b x^ddx\leq b^d(b-a).
$$
Hence
\begin{multline*}
(a^d-b^d)(b-a)-O(\eta(N))\leq\frac{1}{N}\sum_{X<n\leq N+X}\{\theta n\}^d \mathbf{1}_{(a,b)}(\pft{n})-\int_a^b x^ddx\\
\leq (b^d-a^d) (b-a)+O(\eta(N)).
\end{multline*}
Given that
$b^d-a^d\leq d(b-a)$,
we find that
$$
\abs{\sum_{X<n\leq N+X}\{\theta n\}^d \mathbf{1}_{(a,b)}(\pft{n})-\int_a^b x^ddx}\leq d(b-a)^2+O(\eta(N)).
$$
Then splitting the interval $[a,b]$ into 
$O(\sqrt{\eta(N)}^{-1})$ consecutive intervals of size
$\lfloor \sqrt{\eta(N)}\rfloor$, we obtain, for each of these intervals, an error term of size $O(d\eta(N))$, and so in total, an error term of size 
$O(D\sqrt{\eta(N)})$.
\end{proof}

A certain type of sums will appear in the sequel, for which we now give an asymptotic.
\begin{lm}\label{DEL}
Let $0<\alpha,\beta<1$ and
\begin{equation}\label{JAB}
J_N(\alpha,\beta):=\sum_{0<x<N} \frac1{x^{\alpha}(N-x)^{\beta}}.
\end{equation}
Then
$$
J_N(\alpha,\beta)=\begin{cases}
B(1-\alpha,1-\beta)N^{1-\alpha-\beta}+O(N^{-\min(\alpha,\beta)}) &\text{ if $\beta<1$,}\\
N^{-\alpha}\log N+O(N^{-\alpha}) &\text{ if $\beta=1$,}\\
 \zeta(\beta)N^{-\alpha}+O(N^{-\alpha-1+1/\beta})&\text{ if $\beta>1$,}\\
\end{cases}
$$
where $B(\cdot,\cdot)$ denotes the Euler beta function
defined by
$$
B(x,y)=\int_0^1t^{x-1}(1-t)^{y-1}dt
$$
and
$\zeta(\cdot)$ is the Riemann zeta function.
\end{lm}
This can be proven by considering Riemann sums; we omit the
standard details.
The beta function satisfies the following functional equation involving
Euler's gamma function: 
$$B(x,y)=\frac{\Gamma(x)\Gamma(y)}{\Gamma(x+y)}.$$

By induction, we may achieve the following simple lemma.
\begin{lm}
\label{lm:error}
Let $(\alpha_1,\ldots,\alpha_s)\in (0,1)^s$. Then
$$
\sum_{\substack{1\leq u_1,\ldots,u_s\leq n\\\sum_i u_i=n}}\prod_iu_i^{-\alpha_i}=O(n^{s-1-\sum_i\alpha_i}).
$$
Further, let $\epsilon : \N\rightarrow\R_+$ tend
to 0.
Then there exists a sequence $\epsilon'$ depending only on $\epsilon$ that tends to zero such that
$$
\sum_{\substack{1\leq u_1,\ldots,u_s\leq n\\\sum_i u_i=n}}
\epsilon(u_1)\prod_iu_i^{-\alpha_i}=\epsilon'(n)n^{s-1-\sum_i\alpha_i}.
$$
\end{lm}
\begin{proof}
We prove the second part for $s=2$, the rest following
by a simple induction.
Let $K_\delta$ be such that for all $k\geq K_\delta$,
we have
$\epsilon(k)\leq \delta$. Further let $M$ be an upper bound for $\epsilon$.
Then
$$
\sum_{k<n}\epsilon(k)k^{-\alpha_1}(n-k)^{-\alpha_2}
\leq M\sum_{k<K_\delta}k^{-\alpha_1}(n-k)^{-\alpha_2}+\delta \sum_{k<n}k^{-\alpha_1}(n-k)^{-\alpha_2}$$
The right-hand side is
$O(K_\delta^{1-\alpha_1-\alpha_2}+
\delta n^{1-\alpha_1-\alpha_2})$ by Lemma \ref{DEL}.
We have $K_\delta \rightarrow \infty$ (unless $\epsilon(k)=0$ eventually) as $\delta\rightarrow 0$,
but choosing $\delta$ as a sufficiently slowly decaying function of $n$, we can make the error term
as small as $o(n^{1-\alpha_1-\alpha_2})$
as desired.
\end{proof}

For any real number $0\le x\le 1$ and any  integer $1\le j\le k-1$, let
$$
a_j(x)=\max\left(0,x-\frac{j}{k+1}\right), \quad  b_j(x)=\min\left(x,\frac{j}{k+1}\right)
$$
and $I_j(x)$ be the open interval
$$
 I_j(x)=]a_j(x),b_1(x)[.
$$
Let $f_1=\mathbf{1}_{[0,1[}$ and 
$$
f_{j+1}(x)=\int_{a_{j}(x)}^{b_{1}(x)}f_{j}(x-y)dy=\int_{a_1(x)}^{b_{j}(x)}f_{j}(y)dy,\quad 1\le j\le k-1.
$$
Then for any $1\le j\le k-1$ 
\begin{enumerate}

\item[i)] $a_j$ and $b_j$ are piecewise affine.
Further $a_j(x)+b_j(x)=x$.

\item[ii)] $\mu(I_j(x))=b_1(x)-a_j(x)=\max\left(0,\min\left(x,\frac1{k+1},\frac{j+1}{k+1}-x
\right)\right)$.
As a result, $f_j$ is supported on $(0,\frac{j}{k+1})$.

\item[iii)] $f_{j}$ is a non negative, nonzero piecewise polynomial function.
In fact $f_j$ has only finitely many zeros on $(0,j/(k+1))$.

\end{enumerate}
%

\medskip
We will need the following estimate.
\begin{lm}
\label{lm:induction}
Let $(\alpha,\beta)\in (0,1)^2$. Let $\theta>1$ be irrational and $x\in (0,1)$. 
Then for any $j$, we have
$$\sum_{\substack{0<u < N\\\pft{u}\in I_j(x)}}f_j(x-\pft{u})
\frac1{u^{\alpha}}\frac1{(N-u)^{\beta}}=J_N(\alpha,\beta)(f_{j+1}(x)+O(\eta'(N)))$$
where $\eta'$ is a function $ \N\rightarrow\R_+$ that tends to zero
and that depends only on $\theta$.
\end{lm}
\begin{proof}
We decompose the interval of summation $[1,N)$
into subintervals
of some length $m=f(N)$ tending to infinity rather slowly, $m=o(N)$ at any rate, even $m\ll N^{o(1)}$ but not too slowly either;
we fix $m= \lfloor\eta(N)^{-1/2}\rfloor$ for definiteness. We write
$$[1,N)=\bigcup_{0\leq k < \lfloor\frac{N}{m}\rfloor}
(km,(k+1)m]\cup \left( \left\lfloor\frac{N}{m}\right\rfloor m,N\right)
$$
where the last interval has at most $m$ elements.

Let $K=\lfloor\frac{N}{m}\rfloor$.
Let $a=-\alpha$ and $b=-\beta$.
We note that 
$$
\sum_{n\in ( \lfloor\frac{N}{m}\rfloor m,N)}
n^a(N-n)^b\leq m (Km)^a.
$$
Denoting by $S$ the sum to estimate, this implies that
$$
S=\sum_{0\leq k <K}\sum_{\substack{n\in (km,(k+1)m]\\\pft{n}\in I_j(x)}}f_j(x-\{\theta n\})n^a(N-n)^b+O(N^{a+o(1)}).
$$

Also we note that
when $n\in (km,(k+1)m]$, the expression
$n^a(N-n)^b$ may be regarded as approximately constant,
more precisely
$$n^a(N-n)^b=m^{a+b} k^a(K-k)^b
(1+O(1/k))(1+O(1/(K-k))).$$
We may restrict the sum over $k$ to reasonably large $k$, like
between $\sqrt{K}$ and $K-\sqrt{K}$;
indeed,
we have
$$
\sum_{0<u\leq m\sqrt{K}}u^a(N-u)^b\leq (N-N^{1/2+o(1)})^b
N^{a/2+1/2+o(1)}=N^{b+a/2+1/2+o(1)}
$$
which is negligible to $N^{a+b+1}$.
We may argue analogously to discard the sum over 
$k\geq K-\sqrt{K}$.
This way 
$(1+O(1/k))(1+O(1/(K-k))=1+O(1/\sqrt{K})$ for any $k$ considered.
Thus $S$, up to an error
$O(N^{a+b+1}/\sqrt{K}))$, equals
\begin{equation}
\label{eq:sum}
m^{a+b}(1+O(1/\sqrt{K}))\sum_{\sqrt{K}\leq k < K-\sqrt{K}}k^a(K-k)^b\sum_{\substack{n\in (km,(k+1)m]\\\pft{n}\in I_j(x)}}f_j(x-\{\theta n\}).
\end{equation}
We now apply Proposition \ref{prop:genET}
to the inner sum, and by definition of
$f_{j+1}$, we obtain
$$\sum_{\substack{n\in (km,(k+1)m]\\\pft{n}\in I_j(x)}}f_j(x-\{\theta n\})=m(f_{j+1}(x)+O(\eta(\sqrt{m}))).$$
Injecting that in \eqref{eq:sum},
we find that
$$
S=m^{a+b+1}(1+O(1/\sqrt{K}))(f_{j+1}(x)+O(\eta(\sqrt{m})))\sum_{\sqrt{K}\leq k < K-\sqrt{K}}k^a(K-k)^b+O(N^c)
$$
for some $c<a+b+1$.
Now we have from \eqref{JAB}
$$m^{a+b+1}\sum_{\sqrt{K}\leq k <  K-\sqrt{K}}k^a(K-k)^b
=J_N(-a,-b)+O(N^c)$$
by the same arguments as above.
Finally,
upon gathering all error terms together (whereby the term in
$O(\eta(\sqrt{m}))$ provides the largest one),
we obtain the desired conclusion.
\end{proof}

We are now ready to state this subsection's main result.
\begin{lm}\label{LM4}
For any integer $n$,
we have
\begin{equation}\label{equiv}
S_k(n):=\sum_{\substack{0<u_1<\cdots<u_k<n\\ \forall i,\, u_i\in T_{k,\theta}\\n=u_1+\cdots+u_k}}(u_1\cdots u_k)^{-1+1/k}
=\lambda_k f_k(\{\theta n\})
+O\left(\eta''(n)
\right)
\end{equation}
where $\lambda_k=\dfrac{\Gamma(\frac1k)^k}{k!}$
and $\eta''$ is a function decaying to zero (depending on $\theta$ and $k$).
\end{lm}

\begin{proof}
Let 
$$E_k(n):=\sum_{\substack{0<u_1,\cdots, u_k<n\\ \exists i\neq j :  u_i=u_j \\n=u_1+\cdots+u_k}}(u_1\cdots u_k)^{-1+1/k}$$
and
$$
S'_k(n):=\sum_{\substack{0<u_1,\ldots, u_k<n\\ \forall i,\, u_i\in T_{k,\theta}\\n=u_1+\cdots+u_k}}(u_1\cdots u_k)^{-1+1/k},
$$
so that $S'_k(n)=O(E_k(n))+k!S_k(n)$.
We observe that $E_k(n)=O(n^{-1/k})$.
Further, reformulating the diophantine
constraints using the intervals $I_j$,
we have the decomposition
\begin{multline}
\label{eq:initial}
S'_k(n)=
\sum_{\substack{u_1 < n\\ \{\theta u_1\}\in 
I_{k-1}(\{\theta n\})}}
u_1^{-1+1/k}
\sum_{\substack{ u_2 < n-u_1\\ \{\theta u_2\}\in I_{k-2}(\{\theta (n-u_1)\})}}
u_2^{-1+1/k}\quad \cdots\\
\cdots
\sum_{\substack{  u_{k-1} < n-u_1-\cdots-u_{k-2}\\ \{\theta u_{k-1}\}\in I_{1}(\{\theta (n-u_1-\cdots-u_{k-2})\})}}
\big(u_{k-1}(n-u_1-\cdots - u_{k-1})\big)^{-1+1/k}.
\end{multline}
To simplify the notation,
let us denote $n_j=n-u_1-\cdots - u_{k-j}$, thus $n_k=n$ and
$n_j=n_{j+1}-u_{k-j}$.
We shall prove by induction on $j\leq k$ that
\begin{multline}
\label{eq:induction}
S'_k(n)=C_j
\sum_{\substack{0<u_1 < n\\ \{\theta u_1\}\in 
I_{k-1}(\{\theta n\})}}
u_1^{-1+1/k}
\sum_{\substack{0< u_2 < n_1\\ \{\theta u_2\}\in I_{k-2}(\{\theta n_1\})}}
u_2^{-1+1/k}\quad \cdots\\
\cdots
\sum_{\substack{ 0< u_{k-j} < n_{j+1}\\ \{\theta u_{k-j}\}\in I_{j}(\{\theta n_{j+1}\})}}
u_{k-j}^{-1+1/k}(n_{j+1}-u_{k-j})^{-1+j/k}f_j(\{\theta n_{j}\})+\epsilon_j(n)
\end{multline}
where $C_j=\prod_{i=1}^{j-1}B\Big(\frac1k,\frac ik\Big)$
and $\epsilon_j$ tends to 0.
When $j=k$, there is no more summation at all
and \eqref{eq:induction} boils down to $C_{k}f_k(\{\theta n\})+\epsilon_k(n)$,
which is the desired result since
$$
C_k=\prod_{j=1}^{k-1}B\Big(\frac1k,\frac jk\Big) = 
\prod_{j=1}^{k-1}\frac{\Gamma(\frac 1k)\Gamma(\frac jk)}{\Gamma(\frac{j+1}k)} 
=\Gamma\left(\frac 1k\right)^k.
$$

Equation \eqref{eq:initial} is the $j=1$ case. 
We now suppose that \eqref{eq:induction} holds for some
$j\leq k-1$. Let $A_j(n)$ be the main-term of the right-hand side of
\eqref{eq:induction}.
Using Lemma~\ref{lm:induction} on the innermost sum,
and reparametrising by writing $n_{j+1}=v_1$ and $u_i=v_{i+1}$
in the error term,
we find
$$A_j(n)=A_{j+1}(n)+O\Big(\sum_{\substack{v_1,\ldots,v_{k-j}\leq n\\
\sum v_i=n}}\eta'(v_1)v_1^{-1+\frac{j+1}{k}}\prod_{i=2}^{k-j}
v_i^{-1+1/k}\Big).
$$
Now the error term
is certainly $o(1)$ using the fact that $\eta'$ tends to 0
and Lemma~\ref{lm:error}.
This concludes the induction step and therefore the proof of the lemma.
\end{proof}

\subsection{The construction}
We argue by the probabilistic method (see \cite[Chapter~1]{TV} for a brief introduction or \cite{AS} for
a detailed one). 
Let $c>0$ and $\xi_m$, $m\ge1$, be a sequence of independent Boolean random variables such that
$$
\mathbb{P}(\xi_m=1)=\frac{c}{m^{1-1/k}}.
$$
Let $S$ be the random increasing sequence of the $m$'s such that $\xi_m=1$.
This is essentially a sequence of pseudo $k$-th powers.
These objects have been well studied since their introduction by 
Erd\H{o}s and Renyi \cite{ER}.
In particular Goguel \cite{Gog} computed the (almost sure)
density of $kS$
and Deshouillers and Iosifescu \cite{De} found that the density of
$(k+1)S$ is almost surely 1.
Now we let
$A=S\cap T_{k,\theta}$, where $T_{k,\theta}$ was defined 
by equation \eqref{defT}.
From now on we will suppose $\theta$ is irrational;
if $\theta$ is an integer,
$T_{k,\theta}=\N$ so $A=S$ and the previous references apply. The treatment of
this simpler case may still be read
out from our proofs by discarding all the (then vacuous) diophantine 
conditions.
The next proposition implies Proposition \ref{prop:specialcase}.

\vbox{
\begin{prop}\label{pp2}
Almost surely we have
\begin{enumerate}
\item[\rm a)] $\dens{jA}=0$, for any $j=1,\dots,k-1$,

\item[\rm b)] $\dens{(k+1)A}=1$,

\item[\rm c)]  $\dens{kA}=\frac k{k+1}-F_k(c)$ where $F_k(c)$ is a continuous function and increasing from $0$ to $k/(k+1)$ when $c$ is decreasing from $\infty$ to $0$. 
\end{enumerate}
\end{prop}
}
\begin{proof}
a) By an appropriate version of the strong law of large numbers (cf. \cite[chapter III, Theorem 11]{HR}) we know that
with probability 1, $A(x)\sim x^{1/k}$ when $x\to\infty$, thus
for any $1\le j\le k-1$
$$
(jA)(x) \ll x^{j/k},\ \text{as $x$ tends to infinity.}
$$
It follows that $\dens{jA}=0$ almost surely.

\medskip\noindent
b) Let $n$ be a positive integer and observe that $0<\{\theta n\}<1$.
We denote $I(t,k)$ the open interval 
$$
I(t,k)=\left]\max\left(0,\frac{t}{k}-\frac1{k(k+1)}\right), \min\left(\frac{t}k, \frac1{k+1}\right)\right[.
$$
and 
$$
R_{k+1}(n)=\sum_{\substack{0<u_1<\cdots<u_k<u_{k+1}<n\\
n=u_1+\cdots+u_{k+1} \\
\{\theta u_i\}\in I(\{\theta n\},k),\ (1\le i\le k)
}}\xi_{u_1}\dots\xi_{u_k}\xi_{u_{k+1}}.
$$
Then $R_{k+1}(n)>0$ implies that $n\in (k+1)A$. Moreover
$$
\{R_{k+1}(n)=0\}=\bigcap_{\substack{0<u_1<\cdots<u_k<u_{k+1}<n\\ 
n=u_1+\cdots+u_{k+1} \\
\{\theta u_i\}\in I(\{\theta n\},k), \ (1\le i\le k)
}}\{\xi_{u_1}\dots\xi_{u_k}\xi_{u_{k+1}}=0\}.
$$
We denote by $\mathcal{U}(n)$ the set of the ordered $(k+1)$-uples $\underline{u}$ such that
$n=\sum_{i=1}^{k+1}u_i$ and $\{\theta u_i\}\in I(\{\theta n\},k)$, $i=1,\dots,k$.

The events $\mathcal{A}(\underline{u})=\{\xi_{u_1}\dots\xi_{u_k}\xi_{u_{k+1}}=1\}$, $\underline{u}\in\mathcal{U}(n)$, are not necessarily pairwise independent: for distinct $(k+1)$-tuples $\underline{u},\underline{v}$, 
the events $\mathcal{A}(\underline{u})$ and $\mathcal{A}(\underline{v})$ are not independent if and only if
$\underline{u}\sim\underline{v}$,
where the notation $\sim$ means $u_i=v_j$ for some $i,j$. Let
$$
\mu_n=\sum_{\underline{u}\in\mathcal{U}(n)}\mathbb{P}(\mathcal{A}(\underline{u})),
\quad
\Delta_n=\sum_{\substack{\underline{u}\ne \underline{v}\in \mathcal{U}(n)\\\underline{u}\sim\underline{v}}}
\mathbb{P}(\mathcal{A}(\underline{u})\cap \mathcal{A}(\underline{v})).
$$
By Janson's inequality \cite[Theorem 1.28]{TV}
\begin{equation}\label{JI}
\mathbb{P}(R_{k+1}(n)=0)
\le 
\exp\left(-\frac{\mu_n^2}{2(\mu_n+\Delta_n)}\right).
\end{equation}
We firstly have
$$
\mu_n=c^{k+1}
\sum_{\substack{0<u_1<\cdots<u_k<n\\ 
\{\theta u_i\}\in I(\{\theta n\},k)
}}
(u_1\dots u_k(n-u_1-\cdots-u_k))^{-1+1/k}.
$$
The summand in the inner-sum is at least $\big(\frac n{k+1}\big)^{-k+1/k}$, hence
\begin{align*}
\mu_n &\ge c^{k+1} \frac{B_{\theta,I}(n)^k-\binom{k}{2}B_{\theta,I}(n)^{k-1}}{k!} \left(\frac n{k+1}\right)^{-k+1/k} \\
&\ge 
 c^{k+1} \left(B_{\theta,I}(n)^k-\binom{k}{2}B_{\theta,I}(n)^{k-1}\right) n^{-k+1/k} 
\end{align*}
where $I=I(\{\theta n\},k)$. By equation \eqref{eq:Weyl},
$$
\frac{B_{\theta,I}(n)}n\ge\min\left(\frac{\{\theta n\}}k, \frac1{k(k+1)}, \frac{1-\{\theta n\}}k\right)-\eta(n).
$$ 
Hence if $2k\eta(n)<\{\theta n\}<1-2k\eta(n)$, we have 
\begin{equation}
\label{MUN}
\mu_n\ge (1-o(1)) c^{k+1} n^{1/k}\eta(n)^k.
\end{equation}
Now we examine $\Delta_n$.
By a discussion according to the number $s\leq k-1$ of positions
where two distinct $(k+1)$-tuples in $\mathcal{U}(n)$ agree,
and ignoring the diophantine conditions, we get
\begin{equation}
\label{eq:deltan}
\Delta_n\le\sum_{s=1}^{k-1}c^{s+2(k+1-s)}\Delta_{n}(s,k+1-s),
\end{equation}
where
$$
\Delta_{n}(s,r):=
\sum_{\substack{0<u_1,\dots u_s<n\\\sum_{i=1}^{s}u_i<n}}(u_1\dots u_s)^{-1+1/k}\left(
\sum_{\substack{0<v_1,\dots v_r<n\\n=u_1+\cdots+u_s+v_1+\cdots+v_r}}(v_1\dots v_r)^{-1+1/k}
\right)^2.
$$
Applying Lemma \ref{lm:error},
we see that the inner sum is $\ll (n-u_1-\dots-u_s)^{-1+r/k}$. 
For every fixed
tuple $(u_1,\ldots,u_{s-1})$
in the sum above,
we now apply Lemma \ref{DEL}
on the sum
$$
\sum_{u_s<n-\sum_{i=1}^{s-1}u_i}u_s^{-1+1/k}(n-u_1-\dots-u_s)^{-1+r/k}.
$$
If $1-r/k\ge 1/2$, we obtain 
$$
\Delta_{n}(s,r)\ll \log n\sum_{\substack{0<u_1,\dots u_s<n\\n=u_1+\cdots+u_s}}(u_1\dots u_s)^{-1+1/k}\ll \frac{\log n}{n^{1-s/k}}
$$
where we used Lemma \ref{lm:error} for the second inequality.
If $1-r/k<1/2$ then by Lemmas \ref{DEL} and \ref{lm:error} again
\begin{align*}
\Delta_{n}(s,r) & \ll \sum_{\substack{0<u_1,\dots u_s<n\\u_1+\cdots+u_s<n}}(u_1\dots u_s)^{-1+1/k}
(n-u_1-u_2-\cdots-u_s)^{-1+(2r/k-1)}\\
&\ll \begin{cases} n^{-1+(2r/k-1)+s/k} \ll 1 &\text{ if $s\le 2(k-r)$,}\\[1em]
\displaystyle\sum_{\substack{0<u_1,\dots u_t<n\\u_1+\cdots+u_t< n}}(u_1\dots u_t)^{-1+1/k} \ll n^{t/k} &\text{ if $t:=s - 2(k-r)>0$.}\end{cases}
\end{align*}
Notice that if $s+r=k+1$ with $s>0$, then $s-2(k-r)>0$ implies $t=2-s=1$ and $s=1$. We can now inject our upper bounds
for $\Delta_n(s,r)$ in equation \eqref{eq:deltan},
in which the main contribution is given by $s=1$, from the above discussion.
We get
$$
\Delta_n  \ll_k  c^{2k+1} n^{1/k} +O_{k,c}(1).
$$

\medskip
By \eqref{JI} and \eqref{MUN} with the Borel-Cantelli lemma, 
we infer that almost surely, all but finitely many integers $n$ such that
$2k\eta(n)< \{\theta n\}<1-2k\eta(n)$
are sums of $k+1$ members of  $A$
and that $\dens{(k+1)A}=1$ since their complementary set in 
$\mathbb{N}$, namely 
$$
\{n\in\mathbb{N}\,\mid\, 0\le \{\theta n\}\le 2k\eta(n)\}\cup
\{n\in\mathbb{N}\,\mid\, 1-2k\eta(n)\le \{\theta n\}<1\}
$$
has density $0$.

\medskip\noindent
c) Let $n$ such that $0<\{\theta n\}<k/(k+1)$. 
We consider
\begin{equation}\label{RKN}
R_k(n):=k!\sum_{\substack{0<u_1<\cdots<u_k<n\\ u_i\in T_{k,\theta}\\n=u_1+\cdots+u_k}}\xi_{u_1}\dots\xi_{u_k}
\end{equation}
that is the random variable counting the number of representations of $n$ as a sum of $k$  distinct members of $A$.
The key result is Lemma \ref{LM4}.

As in the study of $R_{k+1}(n)$ in the previous paragraph we need to show that the dependency of
the events $\{\xi_{u_1}\dots\xi_{u_k}=1\}$ is not too high. We shall use Landreau's work on sums of $k$ pseudo $k$-th powers (cf. \cite[Lemme 1 (i) and  Lemme 5 (iii)]{La}): 
\begin{align*}
\mathbb{P}(R_k(n)=0)&=\exp\Bigg\{-\sum_{\substack{0<u_1<\cdots<u_k<n\\ u_i\in T_{k,\theta}\\n=u_1+\cdots+u_k}}\mathbb{E}(\xi_{u_1}\dots\xi_{u_k})\Bigg\}+O_k\Big(\frac1{n^{1/k}}\Big)\\
&=e^{-c^{k}S_k(n)}+O_k\Big(\frac1{n^{1/k}}\Big).
\end{align*}
Since $\eta''(t)\to0$ when $t\to\infty$, we deduce from  Lemma \ref{LM4} that 
\begin{equation}\label{pr2}
\mathbb{P}(R_k(n)=0)=e^{-c^k\lambda_kf_k(\{\theta n\})}+o(1).
\end{equation}
When $k/(k+1)\le \{\theta n\} <1$ we clearly have $R_k(n)=0$, hence $\mathbb{P}(R_k(n)=0)=1$.

\bigskip
Let $\zeta_n$, $n\ge1$, be the sequence of Boolean random variables defined by 
$$
\mathbb{P}(\zeta_n=1)=\mathbb{P}(R_k(n)=0),
$$ 
and
$$
X_N=\frac1N\sum_{n=1}^N \zeta_n.
$$
By \eqref{pr2} we have
\begin{equation*}
\sum_{n=1}^N \mathbb{P}(R_k(n)=0)
=\sum_{n=1}^N e^{-c^k\lambda_kf_k(\{\theta n\})} +o(N).
\end{equation*}
Hence, 
\begin{equation*}
\mathbb{E}(X_N)=\frac1N \sum_{n=1}^N \mathbb{P}(R_k(n)=0)
=\sum_{n=1}^N e^{-c^k\lambda_kf_k(\{\theta n\})} +o(1).
\end{equation*}
We get by Theorem \ref{Weyl} and the fact that $f_k$
is supported on $(0,k/(k+1))$ the asymptotic
\begin{equation}\label{EXN}
\mathbb{E}(X_N)\sim \frac 1{k+1} +\int_{0}^{k/(k+1)}e^{-c^k\lambda_kf_k(t)} dt=:\frac 1{k+1}+F_k(c).
\end{equation}

\medskip
We follow the arguments used in the proof of \cite[chapter III, Theorem 4$'$ (iii)]{HR} or alternatively \cite[Section 4]{La} to estimate the variance $\mathbb{V}(X_N)$. We may ignore the diophantine conditions in \eqref{RKN}, the only resulting effect being to increase the related variance.
We finally get  $\mathbb{V}(X_N)=O(N^{-1/k})$ and consequently by \cite[chapter III, lemma 34]{HR} that
$$
\text{with probability $1$,}\quad \lim_{N\to\infty} X_N=\frac 1{k+1}+F_k(c).
$$
Hence almost surely $\dens{kA}=\frac k{k+1}-F_k(c)$. Observing that $f_k$ is a non negative
piecewise polynomial function that has finitely many zeros on $(0,k/(k+1))$,
we see that $F_k(c)$ is a decreasing continuous function
satisfying
$\lim_{c\rightarrow 0} F_k(c)=k/(k+1)$
and
$\lim_{c\rightarrow +\infty} F_k(c)=0$;
this ends the proof of Proposition~\ref{pp2}.
\end{proof}

\section*{Acknowledgments}
The authors are thankful to Georges Grekos for useful conversations.

\bibliographystyle{plain}

\end{document}